\documentclass[12pt]{article}
\pdfoutput=1

\usepackage{amssymb, amsmath, amsthm}
\usepackage{colonequals}
\usepackage{breqn}
\usepackage{bm}
\usepackage{graphicx}
\usepackage{titlesec}
\usepackage{subcaption}

\newtheorem{theorem}{Theorem}[section]

\newtheorem{lemma}[theorem]{Lemma}
\newtheorem{definition}[theorem]{Definition}

\newtheorem{example}[theorem]{Example}

\newcommand{\Po}{\mathcal{P}}
\newcommand{\R}{\mathbb{R}}
\newcommand{\conv}{\text{conv}}

\DeclareMathOperator{\Vol}{vol}
\allowdisplaybreaks

\titleformat{\section}{\normalsize\bfseries}{\thesection}{1.0em}{}

\title{Computing the volume of the convex hull of the
graph of a trilinear monomial using mixed
volumes\footnote{These results were first announced in a 2-page abstract presented at the 5th International Symposium on Combinatorial Optimization (ISCO 2018).} }
\author{{ Emily Speakman and Gennadiy Averkov} \\
{\footnotesize Institute of Mathematical Optimization}\\
{\footnotesize Otto-von-Guericke-Universit\"{a}t, Magdeburg, Germany}
}
\date{\vspace*{-17mm}}

\begin{document}

\maketitle

\begin{abstract} Speakman and Lee (2017) gave a formula for the volume of the convex hull of the graph of a trilinear monomial, $y=x_1x_2x_3$, over a box in the nonnegative orthant, in terms of the upper and lower bounds on the variables.  This was done in the context of using volume as a measure for comparing alternative convexifications to guide the implementation of spatial branch-and-bound for mixed integer nonlinear optimization problems.  Here, we introduce an alternative method for computing this volume, making use of the rich theory of mixed volumes. This new method may lead to a natural approach for considering extensions of the problem.
\end{abstract}

{\center{\bf Key words:} spatial branch-and-bound, convexification, trilinear, polytope, mixed volume. \\[0.0pc]}


\section{Introduction.}\label{intro}
For factorable mixed integer nonlinear optimization (MINLO) problems (see \cite{McCormick76}), spatial branch-and-bound (sBB) is the main algorithmic framework used to find globally optimal solutions (see \cite{Adjiman98}, \cite{Ryoo96}, \cite{Smith99}). At a high level, sBB works by generating a convex relaxation of the problem over a given domain to obtain a bound, and then branching on the domain of a variable and re-convexifying to obtain better bounds.  Therefore, the quality of the convex relaxation obtained is very important in the success of the algorithm.  These problem relaxations are achieved by composing convex relaxations of the simple, low-dimensional functions present in the formulation (over an appropriate domain) and therefore, there has been considerable research into acquiring good convex relaxations of (in particular) multilinear functions, for example see \cite{JMW08}, \cite{Sherali1990}, \cite{Rikun97}, \cite{Meyer04a}, \cite{Meyer04b}.

A reasonable measure for comparing alternative convex relaxations in $\R^n$ is their $n$-dimensional volume.  This was introduced in \cite{Lee94} and has been further developed as a measure to compare polyhedral relaxations (see \cite{LSS2018} and the references therein); to obtain guidance for branching point selection (see \cite{SpeakmanLee_Branching}); and as measure to determine the strength of cuts (see \cite{BCSZ} and \cite{DM}).  In particular, \cite{SpeakmanLee2015} used 4-dimensional volume to compare alternative convexifications of the graph of a trilinear monomial, $y=x_1x_2x_3$, over a box, $x_i \in [a_i, b_i], \; 0 \leq a_i < b_i, \; i=1,2,3$, in the nonnegative orthant.  They calculated the 4-dimensional volume of various natural convexifications of  the graph of $y=x_1x_2x_3$, and used this as a measure to compare the tightness of those convexifications.  A key part of \cite{SpeakmanLee2015} was to compute the 4-dimensional volume of the convex hull.  Define
$$\Po^{3}_H \colonequals \text{conv}\left\{(y,x_1,x_2,x_3) \in \R^4 : y= x_1x_2x_3, \;\; x_i \in [a_i,b_i], \; i=1,2,3 \right\}.$$
 The extreme points of $\Po^3_{H}$ are the eight points that correspond to the $2^3=8$ choices of each $x$-variable at its upper or lower bound (\cite{Rikun97}).  Furthermore, the inequalities of $\Po^3_{H}$ were characterized in \cite{Meyer04a}, \cite{Meyer04b}.

Following \cite{SpeakmanLee2015} (and \cite{Meyer04b}), in the remainder of this paper we will adopt the following convention to label the $x$-variables.  For the variable domains $x_i \in [a_i,b_i]$, $i=1,2,3$, we assume
\begin{equation}
\label{Omega}
 \begin{split}
&0 \leq a_i < b_i \;\text{for}\; i=1,2,3, \quad\text{and} \\
&a_1b_2b_3+b_1a_2a_3 ~\leq~ b_1a_2b_3 + a_1b_2a_3 ~\leq~ b_1b_2a_3 + a_1a_2b_3.
\end{split}
\tag{$\Omega$}
\end{equation}

To see that the latter two inequalities are without loss of generality, let $\mathcal{O}_i \colonequals a_i(b_jb_k) + b_i(a_ja_k)$, where $\{i,j,k\}=\{1,2,3\}$.  Then after re-indexing, we can assume $\mathcal{O}_1 \leq \mathcal{O}_2 \leq \mathcal{O}_3$.  Note that the \eqref{Omega} condition is equivalent to:
\begin{equation}
 \begin{split}
&0 \leq a_i < b_i \;\text{for}\; i=1,2,3, \quad\text{and} \\
&\frac{a_1}{b_1} \leq \frac{a_2}{b_2} \leq \frac{a_3}{b_3}.
\end{split}
\tag{$\Omega^{'}$}
\end{equation}
For more details, see Lemma \ref{lem10.1} (in the appendix).
\newpage

Now we state a main result of \cite{SpeakmanLee2015}.
\begin{theorem}[Theorem 4.1 in \cite{SpeakmanLee2015}]
\label{TheoremPH}
Assuming \eqref{Omega} holds,
\begin{dmath*}
\Vol(\Po^3_H) = (b_1-a_1)(b_2-a_2)(b_3-a_3)\times \\ \left(b_1(5b_2b_3-a_2b_3-b_2a_3-3a_2a_3) + a_1(5a_2a_3-b_2a_3-a_2b_3-3b_2b_3)\right)/24.
\end{dmath*}
\end{theorem}

We can see that $x_2$ and $x_3$ are interchangeable in the volume formula, but $x_1$, when labeled according to \eqref{Omega}, is somehow `special'.  This is exactly what we would expect following the work of \cite{Meyer04a} and \cite{Meyer04b}, which gives the inequality description of $\Po^3_H$.

The original volume proof from \cite{SpeakmanLee2015} works by constructing a triangulation of $\Po^3_H$.  In this work, we present an alternative method for the proof of Theorem~\ref{TheoremPH} using the theory of mixed volumes.   This method, unlike the original, seems to lend itself naturally to some extensions of the problem, as well as being (arguably) a more elegant solution.  Furthermore, it may be possible to use the ideas demonstrated here to compute volumes of other polyhedral relaxations that are interesting to the optimization community.

The paper is structured as follows.  In \S \ref{prelim}, we introduce the concepts needed to state our proof in \S \ref{proof}. In \S \ref{conc}, we consider how this method can be used to extend the work of \cite{SpeakmanLee2015}.  An appendix, \S \ref{A}, contains technical lemmas used for the proof.

\section{Preliminaries.} \label{prelim}

We state definitions, lemmas and theorems that will be important for our proof in the following section.  For the remainder of the paper, we assume the notational convention of using bold lower case letters for vectors, bold upper case letters for matrices, and lower case letters for scalars. Vectors are column vectors, with the transpose of a vector, $\bm{x}$, denoted by $\bm{x}^T$.  $n$ is a positive integer representing the dimension. $\bm{e}_i$, for $i=1\ldots n$, are the standard unit vectors in $\R^n$.  $\mathcal{K}^n$ is the set of all nonempty compact convex sets in $\R^n$. $\R_+$ is the set of nonnegative real numbers.  $\text{vol}(\cdot), \text{aff}(\cdot),$ and $\text{conv}(\cdot)$ refer to the volume, affine hull and convex hull respectively.

We start by stating the definition of mixed volume (see Theorem 5.1.7 in \cite{Schneider2}):

\begin{theorem}\emph{(and Definition)}.  There is a unique, nonnegative function, $V:\underbrace{\mathcal{K}^n \times \cdots \times \mathcal{K}^n}_{n} \rightarrow \R$, the \emph{mixed volume}, which is invariant under permutation of its arguments, such that, for every positive integer $m>0$, one has

$$\Vol(t_1K_1+ t_2K_2+ \dots + t_mK_m) = \sum_{i_1,\dots, i_n =1}^{m}t_{i_1}\dots t_{i_n}V(K_{i_1}, \dots, K_{i_n}),$$
for arbitrary $K_1, \dots K_m \in \mathcal{K}^n$ and $t_1,t_2,\dots,t_n \in \R_{+}$.
\label{thmv}
\end{theorem}

\begin{theorem}
\label{mv2}
The \emph{mixed volume} function satisfies the following properties:
\begin{align*}
\text{(i)}&\qquad \Vol(K_1) = V(K_1,\dots,K_1). \\[2.5mm]
\text{(ii)}&\qquad V(t'K_1'+t''K_1'',K_2,\dots, K_n) = t'V(K'_1, K_2, \dots, K_n) \\
&\qquad \qquad \qquad  \qquad \qquad \qquad \qquad \qquad \qquad\quad + t''V(K''_1, K_2, \dots, K_n),
\end{align*}
for $K_1, \dots K_n \in \mathcal{K}^n$ and $t',t'' \in \R_{+}$.
\end{theorem}

\begin{proof}
(i) is easy to observe by setting $t_1=1$, $t_2 =  \cdots = t_m = 0$ in Theorem~\ref{thmv}, also see page 282 of \cite{Schneider2}.  For (ii), see Equation (5.26), again in \cite{Schneider2}.
\end{proof}

\begin{example}[Mixed Volume]\label{mv}
For the cube $K=[-1,1]^3$, and the regular octahedron $L=\conv(\pm \bm{e}_1, \pm \bm{e}_2, \pm \bm{e}_3)$ consider the behavior of $\Vol(K+ t L)$ as a function of $t \geq0$.

The polytope $K+ t L$ can be decomposed into the following: $K$, eight simplices attached to the vertices of $K$ which altogether comprise $t L$, the triangular prisms attached to the 12 edges of $K$, and the parallelepipeds sitting on the facets of $K$.  For an illustration, see Figure \ref{figmain}.

From this we can directly observe why $\Vol(K+ t L)$ is a cubic polynomial in $t$.  The constant term is $\Vol(K)$.  The cubic term is $t^3 \Vol(L)$.  The parallelepipeds attached to the facets of $K$ yield the linear term, whose coefficient coincides, up to a factor depending on the dimension, with $V(K,K,L)$. The triangular prisms attached to the edges yield the quadratic term, whose coefficient coincides, up to a factor depending on the dimension, with $V(K,L,L)$.

\begin{figure}[h]
\begin{subfigure}{.33\textwidth}
  \centering
  \includegraphics[width=\linewidth]{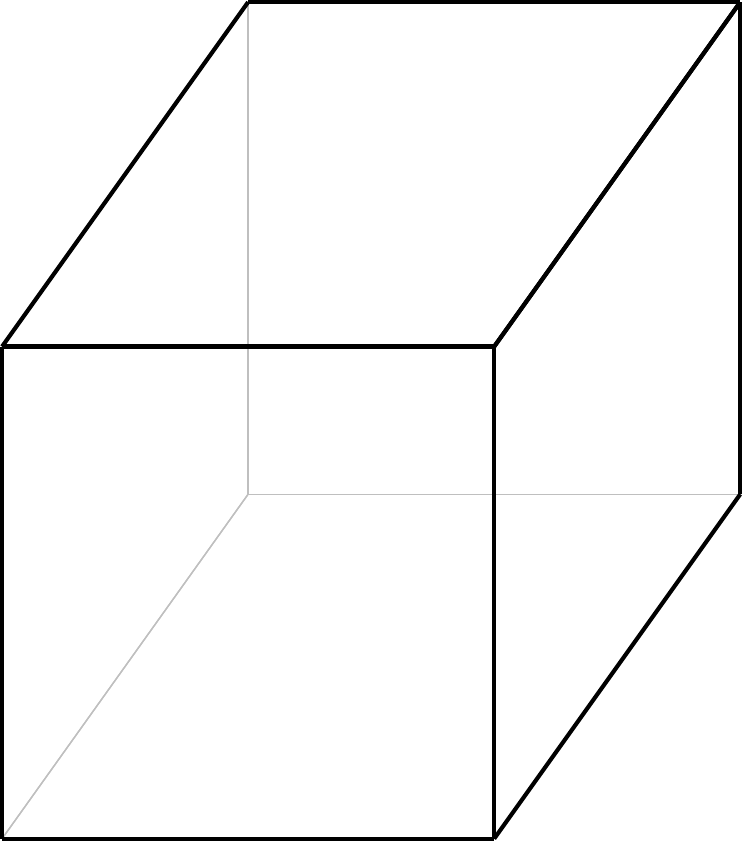}
  \caption{\footnotesize{$K=[-1,1]^3$}}\label{fig0}
\end{subfigure}
\begin{subfigure}{.33\textwidth}
  \centering
  \includegraphics[width=\linewidth]{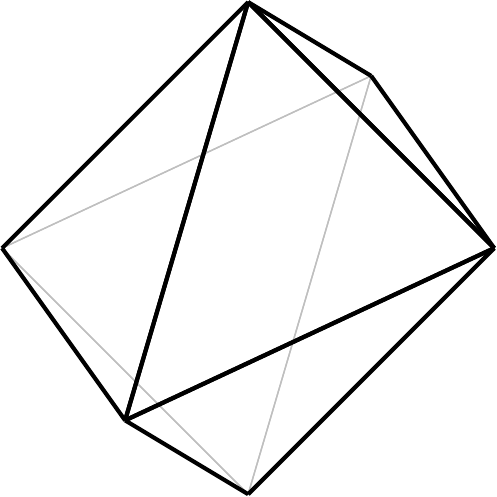}
  \caption{\footnotesize{$L=\conv(\pm \bm{e}_1, \pm \bm{e}_2, \pm \bm{e}_3)$}}\label{fig1}
\end{subfigure}
\begin{subfigure}{.33\textwidth}
  \centering
  \hspace*{40mm}\includegraphics[width=\linewidth]{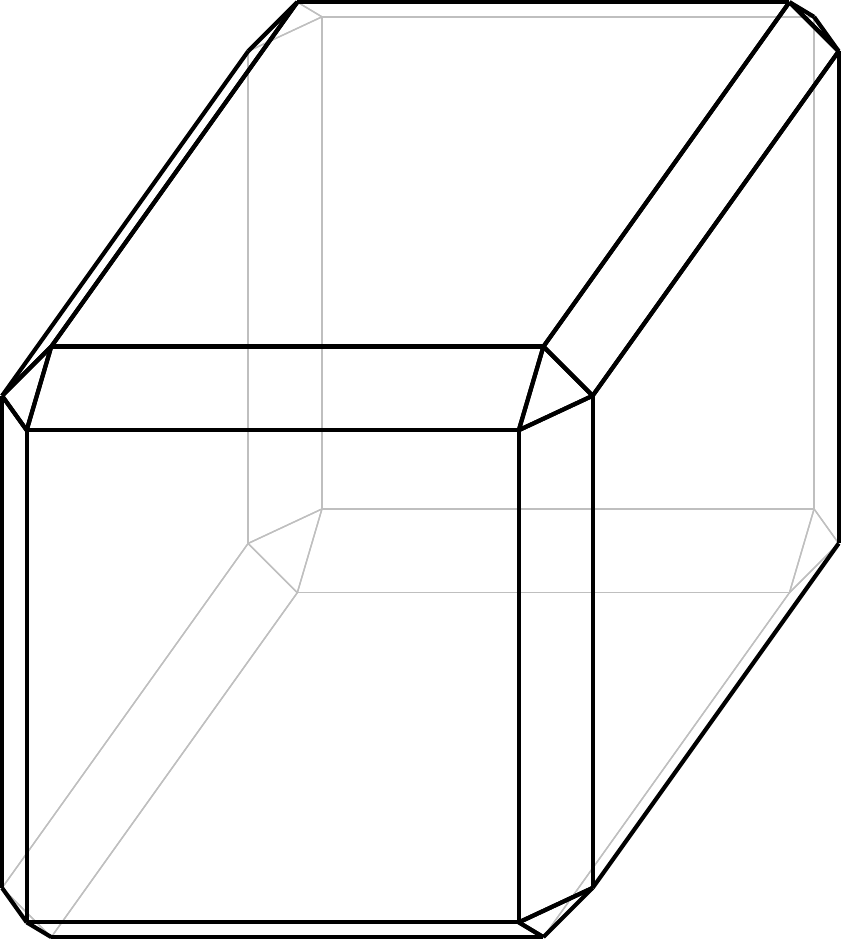}
  \caption{\footnotesize{$K+ t L$, $t>0$}}\label{fig2}
  \end{subfigure}
  \caption{Example \ref{mv}: Mixed Volume}\label{figmain}
\end{figure}

\end{example}

\begin{definition}[Support function]
For $K \in \mathcal{K}^n$, the support function, $h_K:\R^n \rightarrow \R$, is defined by $$h_K(\bm{u})=\sup_{\bm{x}\in K} \bm{x}^T\bm{u}.$$
\end{definition}

\begin{definition}
For a full-dimensional polytope, $P \subseteq  \R^n$, we introduce \begin{align*}\mathcal{U}(P) \colonequals \;\;&\text{the set of all outer facet normals, $\bm{u}$, such that the length}\\ &\text{of $\bm{u}$ is the ($n-1$)-dimensional volume of the respective facet}.\end{align*}
\end{definition}

\begin{theorem}
\label{mixvolT}
Let $P \subseteq \R^n$ be a full-dimensional polytope and let $K \in \mathcal{K}^n$. Then the \emph{mixed volume} of $n-1$ copies of $P$ and one copy of $K$ is given by

$$V(P,P,\dots,P,K) = \frac{1}{n} \sum_{\bm{u} \in \mathcal{U}(P)} h_{K}(\bm{u}).$$
\end{theorem}

The proof follows directly from the literature, see Equation (5.34) of \cite{Schneider2}.  However, note that \cite{Schneider2} doesn't employ the notation $\mathcal{U}(P)$, but rather uses the more general notation for the so-called \emph{mixed-area measure}, $S_{n-1}(P,\cdot)$, which is introduced not only for full-dimensional polytopes but also for arbitrary elements of $\mathcal{K}^n$.   In the special case of
full-dimensional polytopes, $S_{n-1}(P,\cdot)$ is a measure whose support is the set of outer facet normals of $P$, where the measure of an outer facet normal equates to the $(n-1)$-dimensional volume of the respective facet.  To make clear how the notation $S_{n-1}(P,\cdot)$ used in \cite{Schneider2} is related to our notation $\mathcal{U}(P)$, we observe that $\mathcal{U}(P)$ provides the same information as $S_{n-1}(P,\cdot)$, and one has the equality $$\int h_K(\bm{u}) S_{n-1}(P, d\bm{u}) = \sum_{\bm{u} \in \mathcal{U}(P)} h_K(\bm{u}).$$

We derive a formula for $\mathcal{U}(T)$ for an arbitrary tetrahedron $T$.  Consider a $4\times3$ matrix, $\bm{M}$, of the following form:
\begin{equation*}\bm{M}:= \left[\begin{array}{ccc}\bm{\alpha} & \bm{\beta} & \bm{\gamma}   \\1&1&1\end{array} \right] = \left [\begin{array} {ccc}
\alpha_1 & \beta_1 & \gamma_1 \\
\alpha_2 & \beta_2 & \gamma_2\\
\alpha_3 & \beta_3 & \gamma_3 \\
1 & 1 & 1
  \end{array} \right].\end{equation*}
Let $D_i(\bm{\alpha}, \bm{\beta}, \bm{\gamma})$ be the determinant of the $3\times3$ matrix obtained by deleting row $[\alpha_i, \beta_i, \gamma_i]$ from $\bm{M}$.

\begin{lemma}
\label{UT}
Let $T$ be a tetrahedron in $\R^3$, with vertices $\bm{\alpha}$, $\bm{\beta}$, $\bm{\gamma}$, and $\bm{\delta}$, and without loss of generality, label the vertices such that
$$ \det\left[\begin{array}{cccc}\bm{\alpha} & \bm{\beta} & \bm{\gamma} & \bm{\delta}  \\1&1&1&1\end{array}\right] > 0.$$  Then
$${\scriptsize \mathcal{U}(T) = \frac{1}{2}\left \{ \left [\begin{array} {cc}
D_1(\bm{\alpha}, \bm{\beta}, \bm{\gamma}) \\
-D_2(\bm{\alpha}, \bm{\beta}, \bm{\gamma})\\
D_3(\bm{\alpha}, \bm{\beta}, \bm{\gamma})
  \end{array} \right] , \left [\begin{array} {cc}
-D_1(\bm{\delta}, \bm{\alpha}, \bm{\beta}) \\
D_2(\bm{\delta}, \bm{\alpha}, \bm{\beta})\\
-D_3(\bm{\delta}, \bm{\alpha}, \bm{\beta})
  \end{array} \right], \left [\begin{array} {cc}
D_1(\bm{\gamma}, \bm{\delta}, \bm{\alpha}) \\
-D_2(\bm{\gamma}, \bm{\delta}, \bm{\alpha})\\
D_3(\bm{\gamma}, \bm{\delta}, \bm{\alpha})
  \end{array} \right], \left [\begin{array} {cc}
-D_1(\bm{\beta}, \bm{\gamma}, \bm{\delta}) \\
D_2(\bm{\beta}, \bm{\gamma}, \bm{\delta})\\
-D_3(\bm{\beta}, \bm{\gamma}, \bm{\delta})
  \end{array} \right] \right\}}.$$

 \end{lemma}

\begin{proof}

Given that $T$ is a tetrahedron, $\mathcal{U}(T)$ will consist of four vectors, one corresponding to each facet.  Here we determine the vector corresponding to the facet, $\text{conv}(\bm{\alpha}, \bm{\beta}, \bm{\gamma})$.  The remainder of the result then follows by symmetry.

If we assume $\bm{\alpha}=[\alpha_1,\alpha_2,\alpha_3]^T$, $\bm{\beta}=[\beta_1,\beta_2,\beta_3]^T$, $\bm{\gamma}=[\gamma_1,\gamma_2,\gamma_3]^T$, and $\bm{\delta}=[\delta_1,\delta_2,\delta_3]^T$, then the plane, $\text{aff}(\bm{\alpha}, \bm{\beta}, \bm{\gamma})$ can be given by the following equation:
$$\left |\begin{array}{cccc} x_1 & \alpha_1 & \beta_1 & \gamma_1 \\ x_2 & \alpha_2 & \beta_2 & \gamma_2 \\x_3 & \alpha_3 & \beta_3 & \gamma_3 \\ 1&1&1&1\end{array}\right | = 0.$$
Expanding with respect to the first column we obtain:
$$x_1 \left |\begin{array}{ccc} \alpha_2 & \beta_2 & \gamma_2 \\ \alpha_3 & \beta_3 & \gamma_3 \\1&1&1\end{array}\right | - x_2 \left |\begin{array}{ccc} \alpha_1 & \beta_1 & \gamma_1 \\ \alpha_3 & \beta_3 & \gamma_3 \\1&1&1\end{array}\right |+ x_3 \left |\begin{array}{ccc} \alpha_1 & \beta_1 & \gamma_1 \\ \alpha_2 & \beta_2 & \gamma_2 \\1&1&1\end{array}\right | = \left |\begin{array}{ccc} \alpha_1 & \beta_1 & \gamma_1 \\ \alpha_2 & \beta_2 & \gamma_2 \\\alpha_3 & \beta_3 & \gamma_3\end{array}\right |. $$
Therefore, the vector $\bm{u}=[u_1,u_2,u_3]$, with:
\begin{align*}u_1\colonequals \left |\begin{array}{ccc} \alpha_2 & \beta_2 & \gamma_2 \\ \alpha_3 & \beta_3 & \gamma_3 \\1&1&1\end{array}\right |, \;\;
u_2 \colonequals -\left |\begin{array}{ccc} \alpha_1 & \beta_1 & \gamma_1 \\ \alpha_3 & \beta_3 & \gamma_3 \\1&1&1\end{array}\right |,\;\;
u_3 \colonequals \left |\begin{array}{ccc} \alpha_1 & \beta_1 & \gamma_1 \\ \alpha_2 & \beta_2 & \gamma_2 \\1&1&1\end{array}\right |,
\end{align*}
 is a normal vector of the plane, $\text{aff}(\bm{\alpha}, \bm{\beta}, \bm{\gamma})$.  Furthermore, the length of this vector is the area of the parallelogram defined by the vectors $\bm{\beta}-\bm{\alpha}$ and $\bm{\gamma}-\bm{\alpha}$, this parallelogram being twice the area of the facet (triangle). Therefore either $\frac{1}{2}\bm{u}$ or $-\frac{1}{2}\bm{u}$ is in $\mathcal{U}(T)$.  To determine which, we simply consider the remaining vertex $\bm{\delta}$. If
 $$\left |\begin{array}{cccc} \delta_1 & \alpha_1 & \beta_1 & \gamma_1 \\ \delta_2 & \alpha_2 & \beta_2 & \gamma_2 \\\delta_3 & \alpha_3 & \beta_3 & \gamma_3 \\ 1&1&1&1\end{array}\right | < 0,$$
 then $\frac{1}{2}\bm{u} \in \mathcal{U}(T)$, otherwise $-\frac{1}{2}\bm{u} \in \mathcal{U}(T)$.  However, note that
 $$\det\left[\begin{array}{cccc}\bm{\delta} & \bm{\alpha} & \bm{\beta} & \bm{\gamma} \\ 1&1&1&1\end{array} \right] = -\det\left[\begin{array}{cccc}\bm{\alpha} & \bm{\beta} & \bm{\gamma} & \bm{\delta}  \\1&1&1&1\end{array} \right]. $$
 Therefore, we can label the vertices such that
 $$ \det\left[\begin{array}{cccc}\bm{\alpha} & \bm{\beta} & \bm{\gamma} & \bm{\delta}  \\1&1&1&1\end{array}\right] > 0,$$
 and we have $\frac{1}{2}\bm{u} \in \mathcal{U}(T)$ where, using our notation, $\bm{u} = \left [\begin{array} {cc}
D_1(\bm{\alpha}, \bm{\beta}, \bm{\gamma}) \\
-D_2(\bm{\alpha}, \bm{\beta}, \bm{\gamma})\\
D_3(\bm{\alpha}, \bm{\beta}, \bm{\gamma})
  \end{array} \right]$.

  Symmetric arguments can be made to obtain each of the other elements of $\mathcal{U}(T)$.

\end{proof}

\section{Alternative proof of Theorem~\ref{TheoremPH}.} \label{proof}

\paragraph{Note that throughout this proof, we assume \eqref{Omega}}

For convenience, we will also assume throughout that $a_1, a_2, a_3 >0$.  By an identical continuity argument to the one made in \cite{SpeakmanLee2015} (page 7), we are then immediately able to claim that the resulting formula also holds for any $a_i=0$.  In fact, the only time this impacts our proof is in the case $a_3=0$, which in turn, by \eqref{Omega}, implies $a_2=0$ and $a_1=0$.  Therefore, the volume calculation becomes much simpler and we could easily construct a separate proof for this case.  However, because of the continuity argument, this is not necessary.

To compute the four dimensional volume of $\Po^3_H$, we first note that the extreme points of $\Po^3_H$, lie in two parallel hyperplanes.  Four points lie in the hyperplane $x_3=a_3$ and four points lie in the hyperplane $x_3=b_3$.  See Figure~\ref{cube} for an illustration adapted from \cite{SpeakmanLee2015}. In this way we can think of $\Po^3_H$ as the convex hull of two (3 dimensional) tetrahedra sat in $\R^4$, we define these as follows:

\begin{align*}Q&:=\text{conv}\left\{
\left [\begin{array} {c}
b_1b_2a_3 \\
b_1\\
b_2
  \end{array} \right], \;\; \left [\begin{array} {c}
a_1a_2a_3\\
a_1\\
a_2
  \end{array} \right], \;\; \left [\begin{array} {c}
b_1a_2a_3\\
b_1\\
a_2
  \end{array} \right],\; \;\left [\begin{array} {c}
a_1b_2a_3 \\
a_1\\
b_2
  \end{array} \right]
 \right\}, \\[12pt]
R&:=\text{conv}\left\{
\left [\begin{array} {c}
b_1b_2b_3 \\
b_1\\
b_2
  \end{array} \right], \;\; \left [\begin{array} {c}
a_1a_2b_3\\
a_1\\
a_2
  \end{array} \right], \;\; \left [\begin{array} {c}
b_1a_2b_3\\
b_1\\
a_2
  \end{array} \right],\; \;\left [\begin{array} {c}
a_1b_2b_3 \\
a_1\\
b_2
  \end{array} \right]
  \right\}.\end{align*}

The 4-dimensional volume of $\Po_H^3$ can therefore be calculated via an integral of the 3-dimensional volumes of parallel cross-sections of $\Po_H^3$ as $x_3$ varies from $a_3$ to $b_3$:
\begin{equation*}\Vol(\Po^3_H)=  \int_{a_3}^{b_3} \Vol \left ( \frac{b_3-t}{b_3-a_3}Q + \frac{t-a_3}{b_3-a_3}R \right )dt. \end{equation*} It is clear that when $t=a_3$ we have the volume of $Q$ and when $t=b_3$ we have the volume of $R$.

\allowdisplaybreaks[1]
Using Theorem~\ref{thmv} and Theorem~\ref{mv2}, we are able to write the volume under the integral (of the cross section of $\Po^3_H$) in terms of mixed volumes (of $Q$ and $R$), and obtain:
\begin{align*}\Vol(\Po^3_H)&=  \int_{a_3}^{b_3} \Vol \left ( \frac{b_3-t}{b_3-a_3}Q + \frac{t-a_3}{b_3-a_3}R \right )dt \\[5pt]
 &=(b_3-a_3)^{-3} \int_{a_3}^{b_3} \Vol\left((b_3-t)Q + (t-a_3)R\right) dt \\[5pt]
 &=(b_3-a_3)^{-3} \int_{a_3}^{b_3}\left ( (b_3-t)^3 V(Q,Q,Q) + 3(b_3-t)^2(t-a_3) V(Q,Q,R)\right.\\
 &\qquad \qquad \quad \;\; \left.+ 3(b_3-t)(t-a_3)^2V(Q,R,R) + (t-a_3)^3V(R,R,R)\right)dt \\[5pt]
 &=(b_3-a_3)^{-3} \int_{a_3}^{b_3}\left ( (b_3-t)^3 \Vol(Q) + 3(b_3-t)^2(t-a_3) V(Q,Q,R)\right.\\
 &\qquad \qquad \qquad \;\; \left.+ 3(b_3-t)(t-a_3)^2V(Q,R,R) + (t-a_3)^3\Vol(R)\right)dt, \end{align*} where $V(Q,R,R)$ is the mixed volume of one copy of $Q$ and two copies of $R$.  Likewise $V(Q,Q,R)$ is the mixed volume of one copy of $R$ and two copies of $Q$.

\begin{figure}[h]
  \centering
  \includegraphics[width=0.5\textwidth]{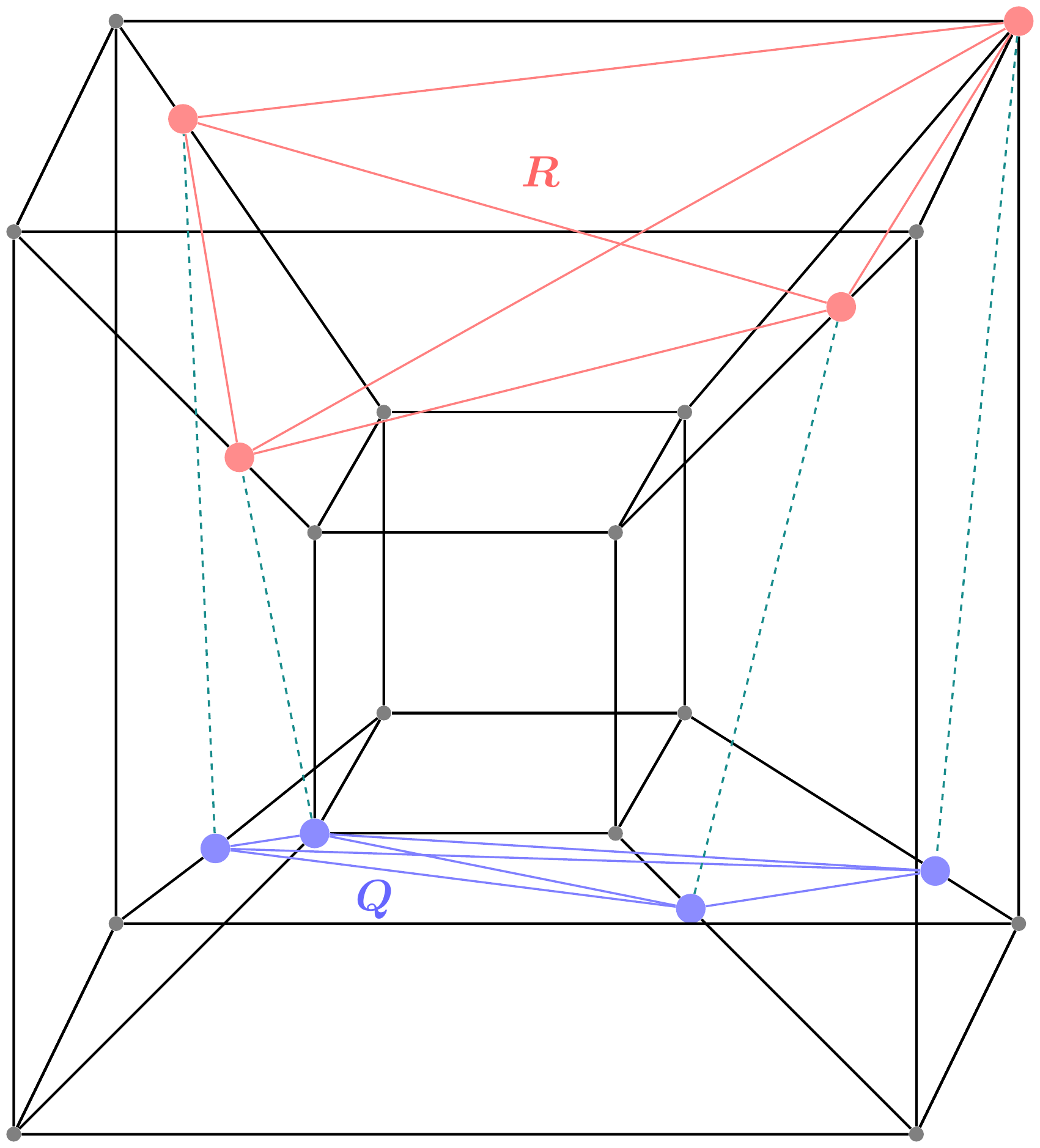}
  \caption{Schlegel diagram of $\Po^3_H$ embedded into the 4-dimensional hypercube}\label{cube}
\end{figure}

The volume of $Q$ and of $R$ can be computed via a simple determinant calculation.  Therefore, we state the following two lemmas without proof.
\begin{lemma}
\label{simplexP}
$$\Vol(Q)=\frac{a_3(b_1-a_1)^2(b_2-a_2)^2}{6}.  $$
\end{lemma}

\begin{lemma}
\label{simplexQ}
$$\Vol(R)=\frac{b_3(b_1-a_1)^2(b_2-a_2)^2}{6}.  $$
\end{lemma}

Now we know $\Vol(Q)$ and $\Vol(R)$ in terms of the parameters, $a_1$, $a_2$, $a_3$, $b_1$, $b_2$, $b_3$, all that is required is to find the relevant mixed volumes and complete the integration.

Directly from Lemma \ref{UT} we obtain:

{\small
\begin{align*}
\mathcal{U}(Q) = &\frac{(b_1-a_1)(b_2-a_2)}{2}\;\left\{ \left [\begin{array} {cc}
1 \\
-a_2a_3\\
-b_1a_3
  \end{array} \right] ,\; \left [\begin{array} {cc}
1 \\
-b_2a_3\\
-a_1a_3
  \end{array} \right], \;
   \left [\begin{array} {cc}
-1 \\
b_2a_3\\
b_1a_3
  \end{array} \right],\; \left [\begin{array} {cc}
-1 \\
a_2a_3\\
a_1a_3
  \end{array} \right] \right\}, \\[12pt]
\mathcal{U}(R) = &\frac{(b_1-a_1)(b_2-a_2)}{2}\;\left\{ \left [\begin{array} {cc}
1 \\
-a_2b_3\\
-b_1b_3
  \end{array} \right] ,\; \left [\begin{array} {cc}
1 \\
-b_2b_3\\
-a_1b_3
  \end{array} \right],  \;\left [\begin{array} {cc}
-1 \\
b_2b_3\\
b_1b_3
  \end{array} \right], \;\left [\begin{array} {cc}
-1 \\
a_2b_3\\
a_1b_3
  \end{array} \right] \right\}.
  \end{align*}}
Then, using Theorem~\ref{mixvolT} we have \begin{equation*}V(Q,Q,R) = \frac{1}{3} \sum_{\bm{u} \in \mathcal{U}(Q)} h_R(\bm{u}) \end{equation*} and  \begin{equation*}V(Q,R,R) = \frac{1}{3} \sum_{\bm{u} \in \mathcal{U}(R)} h_Q(\bm{u}), \end{equation*} where $h_R(\bm{u})=\max_{\bm{x}\in R} \bm{x}^T\bm{u}$ is the support function of $R$, and $h_Q(\bm{u})=\max_{\bm{x}\in Q} \bm{x}^T\bm{u}$ is the support function of $Q$.  Because $Q$ and $R$ are tetrahedra, it is simple to compute $h_R(\bm{u})$ or $h_Q(\bm{u})$ for some vector $\bm{u}$.  For each extreme point, $\bm{x}$, we need only to compute $\bm{x}^T\bm{u}$ and pick the maximum of these four values. The following lemmas (with proofs stated in the appendix) are used to obtain the appropriate maximum values.

\begin{lemma}
\label{lem1}
\begin{align*}
 z_1&\colonequals \max\left\{b_1b_2b_3-b_1a_2a_3-b_1b_2a_3, \; a_1a_2b_3-a_1a_2a_3-b_1a_2a_3,\right.\\
   &\qquad  \qquad \qquad \qquad\qquad \qquad\;\; \left.b_1a_2b_3-2b_1a_2a_3, \; a_1b_2b_3-b_1b_2a_3-a_1a_2a_3\right\}\\
    &= b_1b_2b_3-b_1a_2a_3-b_1b_2a_3.
   \end{align*}
\end{lemma}

\begin{lemma}
\label{lem2}
\begin{align*}
z_2&\colonequals\max\left\{b_1b_2b_3-a_1b_2a_3-b_1b_2a_3, \; a_1a_2b_3-a_1b_2a_3-a_1a_2a_3,\right.\\
    &\qquad  \qquad \qquad \qquad\qquad\qquad\;\; \left.b_1a_2b_3-b_1b_2a_3-a_1a_2a_3, \; a_1b_2b_3-2a_1b_2a_3\right\}\\
   &= b_1b_2b_3-a_1b_2a_3-b_1b_2a_3.
   \end{align*}
\end{lemma}

\begin{lemma}
\label{lem3}
\begin{align*}
 z_3&\colonequals\max\left\{2b_1b_2a_3-b_1b_2b_3, \; a_1b_2a_3+b_1a_2a_3-a_1a_2b_3,\right.\\
    &\qquad  \qquad \qquad \qquad\quad\;\left.b_1a_2a_3+b_1b_2a_3-b_1a_2b_3, \; a_1b_2a_3+b_1b_2a_3-a_1b_2b_3\right\}  \\
   & = a_1b_2a_3+b_1b_2a_3-a_1b_2b_3.
   \end{align*}
\end{lemma}

\begin{lemma}
\label{lem4}
\begin{align*}
z_4&\colonequals\max\left\{a_1b_2a_3+b_1a_2a_3-b_1b_2b_3, \;2a_1a_2a_3-a_1a_2b_3,\right.\\
   &\qquad  \qquad \qquad \qquad\quad \left.a_1a_2a_3+b_1a_2a_3-b_1a_2b_3, \; a_1a_2a_3+a_1b_2a_3-a_1b_2b_3\right\}\\
   & = 2a_1a_2a_3-a_1a_2b_3.
   \end{align*}
\end{lemma}

\begin{lemma}
\label{lem5}
\begin{align*}
 z_5&\colonequals\max\left\{b_1b_2a_3-b_1a_2b_3-b_1b_2b_3, \; a_1a_2a_3-a_1a_2b_3-b_1a_2b_3,\right.\\
   &\qquad  \qquad \qquad \qquad\qquad\qquad\;\; \left.b_1a_2a_3-2b_1a_2b_3,\; a_1b_2a_3-a_1a_2b_3-b_1b_2b_3\right\}\\
   &  = a_1a_2a_3-a_1a_2b_3-b_1a_2b_3.
   \end{align*}
\end{lemma}

\begin{lemma}
\label{lem6}
\begin{align*}
 z_6&\colonequals\max\left\{b_1b_2a_3-a_1b_2b_3-b_1b_2b_3, \; a_1a_2a_3-a_1a_2b_3-a_1b_2b_3,\right.\\
   &\qquad  \qquad \qquad \qquad\qquad\qquad\;\;\left.b_1a_2a_3-b_1b_2b_3-a_1a_2b_3, \;a_1b_2a_3-2a_1b_2b_3\right\} \\
   & = a_1a_2a_3-a_1a_2b_3-a_1b_2b_3.
   \end{align*}
\end{lemma}

\begin{lemma}
\label{lem7}
\begin{align*}
 z_7&\colonequals\max\left\{2b_1b_2b_3-b_1b_2a_3,\; a_1b_2b_3+b_1a_2b_3-a_1a_2a_3,\right.\\
   & \qquad  \qquad \qquad \qquad\quad\;\left.b_1a_2b_3+b_1b_2b_3-b_1a_2a_3, \;a_1b_2b_3+b_1b_2b_3-a_1b_2a_3\right\}\\
   & = 2b_1b_2b_3-b_1b_2a_3.
   \end{align*}
\end{lemma}

\begin{lemma}
\label{lem8}
\begin{align*}
z_8&\colonequals\max\left\{a_1b_2b_3+b_1a_2b_3-b_1b_2a_3, \; 2a_1a_2b_3-a_1a_2a_3,\right.\\
   &\qquad \qquad \qquad \qquad\quad \left.a_1a_2b_3+b_1a_2b_3-b_1a_2a_3, \;a_1a_2b_3+a_1b_2b_3-a_1b_2a_3\right\}\\
   &= a_1a_2b_3+b_1a_2b_3-b_1a_2a_3.
   \end{align*}
\end{lemma}

We obtain the following:
\begin{align*}
V(Q,Q,R) &= \frac{1}{3} \sum_{\bm{u} \in \mathcal{U}(Q)} h_R(\bm{u}) \\
    &= \frac{1}{3} \times \frac{(b_1-a_1)(b_2-a_2)}{2} \times \Big( z_1+z_2+z_3+z_4\Big),
    \end{align*}

 %where:
%\begin{align*} z_1&\colonequals \max\left\{b_1b_2b_3-b_1a_2a_3-b_1b_2a_3, \; a_1a_2b_3-a_1a_2a_3-b_1a_2a_3,\right.\\
%   &\qquad  \qquad \qquad \qquad\qquad \qquad\;\; \left.b_1a_2b_3-2b_1a_2a_3, \; a_1b_2b_3-b_1b_2a_3-a_1a_2a_3\right\} \\[1mm]
%    z_2&\colonequals\max\left\{b_1b_2b_3-a_1b_2a_3-b_1b_2a_3, \; a_1a_2b_3-a_1b_2a_3-a_1a_2a_3,\right.\\
%    &\qquad  \qquad \qquad \qquad\qquad\qquad\;\; \left.b_1a_2b_3-b_1b_2a_3-a_1a_2a_3, \; a_1b_2b_3-2a_1b_2a_3\right\}  \\[1mm]
%    z_3&\colonequals\max\left\{2b_1b_2a_3-b_1b_2b_3, \; a_1b_2a_3+b_1a_2a_3-a_1a_2b_3,\right.\\
%    &\qquad  \qquad \qquad \qquad\quad\;\left.b_1a_2a_3+b_1b_2a_3-b_1a_2b_3, \; a_1b_2a_3+b_1b_2a_3-a_1b_2b_3\right\}  \\[1mm]
%   z_4&\colonequals\max\left\{a_1b_2a_3+b_1a_2a_3-b_1b_2b_3, \;2a_1a_2a_3-a_1a_2b_3,\right.\\
%   &\qquad  \qquad \qquad \qquad\quad \left.a_1a_2a_3+b_1a_2a_3-b_1a_2b_3, \; a_1a_2a_3+a_1b_2a_3-a_1b_2b_3\right\}. \end{align*}

 Using Lemmas \ref{lem1}-\ref{lem4}, we therefore have:
    \begin{align*}
 V(Q,&Q,R) \\
   %&= \frac{(b_1-a_1)(b_2-a_2)}{6} \times\left(\right. b_1b_2b_3-b_1a_2a_3-b_1b_2a_3+ b_1b_2b_3-a_1b_2a_3\\
   %& \quad\quad \; -b_1b_2a_3 + a_1b_2a_3+b_1b_2a_3-a_1b_2b_3 + 2a_1a_2a_3-a_1a_2b_3\left.\right) \\[3mm]
&= \frac{(b_1-a_1)(b_2-a_2)\left((b_1-a_1)(b_2b_3-a_2a_3)+(b_3-a_3)(b_1b_2-a_1a_2)\right)}{6}.
   \end{align*}

Following a similar argument we compute:
\begin{align*}V(Q,R,R) &= \frac{1}{3} \sum_{\bm{u} \in \mathcal{U}(R)} h_Q(\bm{u}) \\
                         &= \frac{1}{3}\times \frac{(b_1-a_1)(b_2-a_2)}{2} \times \Big(z_5+z_6+z_7+z_8\Big)
\end{align*}

% where:
% \begin{align*}
% z_5&\colonequals\max\left\{b_1b_2a_3-b_1a_2b_3-b_1b_2b_3, \; a_1a_2a_3-a_1a_2b_3-b_1a_2b_3,\right.\\
%   &\qquad  \qquad \qquad \qquad\qquad\qquad\;\; \left.b_1a_2a_3-2b_1a_2b_3,\; a_1b_2a_3-a_1a_2b_3-b_1b_2b_3\right\} \\[1mm]
%    z_6&\colonequals\max\left\{b_1b_2a_3-a_1b_2b_3-b_1b_2b_3, \; a_1a_2a_3-a_1a_2b_3-a_1b_2b_3,\right.\\
%   &\qquad  \qquad \qquad \qquad\qquad\qquad\;\;\left.b_1a_2a_3-b_1b_2b_3-a_1a_2b_3, \;a_1b_2a_3-2a_1b_2b_3\right\}  \\[1mm]
%    z_7&\colonequals\max\left\{2b_1b_2b_3-b_1b_2a_3,\; a_1b_2b_3+b_1a_2b_3-a_1a_2a_3,\right.\\
%   & \qquad  \qquad \qquad \qquad\quad\;\left.b_1a_2b_3+b_1b_2b_3-b_1a_2a_3, \;a_1b_2b_3+b_1b_2b_3-a_1b_2a_3\right\}  \\[1mm]
%   z_8&\colonequals\max\left\{a_1b_2b_3+b_1a_2b_3-b_1b_2a_3, \; 2a_1a_2b_3-a_1a_2a_3,\right.\\
%   &\qquad \qquad \qquad \qquad\quad \left.a_1a_2b_3+b_1a_2b_3-b_1a_2a_3, \;a_1a_2b_3+a_1b_2b_3-a_1b_2a_3\right\}.
%\end{align*}

 Using Lemmas \ref{lem5}-\ref{lem8}, we therefore have:
 \begin{align*}
   V(Q,&R,R)\\
   %&= \frac{(b_1-a_1)(b_2-a_2)}{6} \times\left(\right. a_1a_2a_3-a_1a_2b_3-b_1a_2b_3+ a_1a_2a_3 -a_1a_2b_3\\
   % &\quad\quad  -a_1b_2b_3    + 2b_1b_2b_3-b_1b_2a_3 + a_1a_2b_3+b_1a_2b_3-b_1a_2a_3\left.\right) \\[3mm]
   &= \frac{(b_1-a_1)(b_2-a_2)\left((b_1-a_1)(b_2b_3-a_2a_3)+(b_3-a_3)(b_1b_2-a_1a_2)\right)}{6}.
   \end{align*}

It is interesting to note that in this special case we have $V(Q,Q,R) =V(Q,R,R)$.

Now that we have calculated $V(Q,Q,R), V(Q,R,R), \Vol(Q)$ and $\Vol(R)$ in terms of the parameters $a_1,a_2,a_3,b_1,b_2,b_3$, all that remains is to compute the necessary integral.  In doing so we obtain:
{ \begin{align*}\Vol(\Po^3_H)&= (b_3-a_3)^{-3} \int_{a_3}^{b_3} \left((b_3-t)^3 \Vol(Q) + 3(b_3-t)^2(t-a_3) V(Q,Q,R)\right.\\
 &\qquad \qquad \quad \;\;\left.+ 3(b_3-t)(t-a_3)^2V(Q,R,R) + (t-a_3)^3\Vol(R)\right)\;dt \\[4mm]
 &=(b_1-a_1)(b_2-a_2)(b_3-a_3)\times \\
 &\hspace*{-9mm}\left(b_1(5b_2b_3-a_2b_3-b_2a_3-3a_2a_3) + a_1(5a_2a_3-b_2a_3-a_2b_3-3b_2b_3)\right)/24.
 \end{align*}}
This gives us the same formula as obtained via a triangulation method in \cite{SpeakmanLee2015}.

%\subsection{A short note on the presentation of the volume formula}
%
%Let $\Delta_i=b_i-a_i$ for $i=1,2,3$.  Then we can write the volume of $\Po^3_H$ as follows.
%$$\Vol(\Po^3_H)=\frac{1}{24} \left( 5a_1\Delta_2\Delta_3 + 4\Delta_1a_2\Delta_3+4\Delta_1\Delta_2a_3 +\Delta_1\Delta_2\Delta_3  \right). $$
%Moreover, if $\Delta_1=\Delta_2=\Delta_3=1$, then the volume is simply a linear function of the variable lower bounds.

\section{Further Work.} \label{conc}

We have used mixed volume theory to obtain an alternative proof of Theorem~\ref{TheoremPH}.  To our knowledge, this is the first demonstration of this kind, showing how mixed volumes can be directly applied to perform computations relevant to the optimization community.  This new method of proof allows us to consider extensions of the problem that previously seemed too complex or tedious.  For example, throughout this work we have assumed that the bounds on the variables are nonnegative, however, a natural extension of the problem is to allow the variables to have mixed-sign domains.  Using the techniques of this paper, we are hopeful that volume formulae will also be obtained for this case and furthermore we expect these will also result in trilinear expressions.  This is the focus of our current work.

A different, yet still natural extension, is to ask if we can compute the volume of the convex hull of the graph of a \emph{general} multilinear term (over a box).  In the work of \cite{SpeakmanLee2015}, the jump to the case of general $n$ seemed to be computationally unrealistic.  Now, with this alternative method, it is still not certain that the leap will be possible, but we do have a natural approach to this problem.  In the case of general $n$, we can write down the volume recursively as:
$$\Vol(\Po_H^n)=  \int_{a_n}^{b_n} \Vol \left ( \frac{b_n-t}{b_n-a_n}Q_{n-1} + \frac{t-a_n}{b_n-a_n}R_{n-1} \right )dt,$$
where each of the polytopes $Q_{n-1}$ and $R_{n-1}$, are no longer tetrahedra, but \emph{are} closely related to $\Po_H^{n-1}$ (the extreme points of both are the extreme points of $\Po_H^{n-1}$ with a scaling applied to the first component).  Being able to express the volume of $\Po_H^n$ in terms of the volumes of $Q_{n-1}$ and $R_{n-1}$ is hopefully the first step in obtaining more general results, beginning with $n=4$.

\section{Acknowledgements.}

The authors gratefully acknowledge support from the Deutsche Forschungsgemeinschaft (DFG, German Research Foundation) - 314838170, GRK 2297 MathCoRe.

\section{Bibliography.}
  \bibliographystyle{elsarticle-num}
 \bibliography{Speakman_Averkov_MixedVolume_revision_Arxiv}

 %\appendix

 \section{Appendix.}
\label{A}

We state a simple lemma, with proof, that will be useful in proving the subsequent lemmas.  Recall that throughout, we assume \eqref{Omega}.

\begin{lemma}[See also Lemma 10.1 from \cite{SpeakmanLee2015}]
\label{lem10.1}
\begin{equation*}
 \eqref{Omega} \Longleftrightarrow  \;(\Omega^{''}) \; \begin{cases} 0 \leq a_i < b_i \;\text{for}\; i=1,2,3,  \\
b_1a_2-a_1b_2 \geq 0,\; b_1a_3-a_1b_3 \geq 0, \; b_2a_3-a_2b_3 \geq 0.\end{cases}
\end{equation*}

\end{lemma}

\begin{proof}
$\eqref{Omega} \Rightarrow b_1a_2b_3 + a_1b_2a_3 - a_1b_2b_3 -b_1a_2a_3 =  (b_3-a_3)(b_1a_2-a_1b_2)\geq 0$.
This implies $b_1a_2-a_1b_2 \geq 0$, because $b_3-a_3 > 0$.

Conversely, if $(\Omega^{''})$ holds, then $(b_3-a_3)(b_1a_2-a_1b_2)\geq0$ and we obtain $b_1a_2b_3 + a_1b_2a_3 \geq  a_1b_2b_3 + b_1a_2a_3$.

Identical arguments can be made for $b_1a_3-a_1b_3 \geq 0$ and $b_2a_3-a_2b_3 \geq 0$. \end{proof}

We now state the proofs of the four lemmas used in calculating $V(Q,Q,R)$:

%
%\begin{lemma}
%\label{lem1}
%\begin{align*}
% z_1&\colonequals \max\left\{b_1b_2b_3-b_1a_2a_3-b_1b_2a_3, \; a_1a_2b_3-a_1a_2a_3-b_1a_2a_3,\right.\\
%   &\qquad  \qquad \qquad \qquad\qquad \qquad\;\; \left.b_1a_2b_3-2b_1a_2a_3, \; a_1b_2b_3-b_1b_2a_3-a_1a_2a_3\right\}\\
%    &= b_1b_2b_3-b_1a_2a_3-b_1b_2a_3.
%   \end{align*}
%\end{lemma}

\begin{proof}[Proof of Lemma \ref{lem1}]
\begin{align*}
b_1b_2b_3-b_1a_2a_3-b_1b_2a_3 -(a_1a_2b_3-a_1a_2a_3-b_1a_2a_3)&\\
=(b_1b_2-a_1a_2)(b_3-a_3) &\geq0. \\
b_1b_2b_3-b_1a_2a_3-b_1b_2a_3 -(b_1a_2b_3-2b_1a_2a_3)= b_1(b_2-a_2)(b_3-a_3) &\geq0. \\
b_1b_2b_3-b_1a_2a_3-b_1b_2a_3 -(a_1b_2b_3-b_1b_2a_3-a_1a_2a_3)&\\
= (b_2b_3-a_2a_3)(b_1-a_1) &\geq0.
\end{align*}
\end{proof}

%\begin{lemma}
%\label{lem2}
%\begin{align*}
%z_2&\colonequals\max\left\{b_1b_2b_3-a_1b_2a_3-b_1b_2a_3, \; a_1a_2b_3-a_1b_2a_3-a_1a_2a_3,\right.\\
%    &\qquad  \qquad \qquad \qquad\qquad\qquad\;\; \left.b_1a_2b_3-b_1b_2a_3-a_1a_2a_3, \; a_1b_2b_3-2a_1b_2a_3\right\}\\
%   &= b_1b_2b_3-a_1b_2a_3-b_1b_2a_3.
%   \end{align*}
%\end{lemma}

\begin{proof}[Proof of Lemma \ref{lem2}]
\begin{align*}
b_1b_2b_3-a_1b_2a_3-b_1b_2a_3 -(a_1a_2b_3-a_1b_2a_3-a_1a_2a_3)&\\
=(b_1b_2-a_1a_2)(b_3-a_3) &\geq0. \\
b_1b_2b_3-a_1b_2a_3-b_1b_2a_3 -(b_1a_2b_3-b_1b_2a_3-a_1a_2a_3)&\\
= (b_1b_3-a_1a_3)(b_2-a_2)  &\geq0. \\
b_1b_2b_3-a_1b_2a_3-b_1b_2a_3 -(a_1b_2b_3-2a_1b_2a_3) = b_2(b_1-a_1)(b_3-a_3) &\geq0.
\end{align*}
\end{proof}

%\begin{lemma}
%\label{lem3}
%\begin{align*}
% z_3&\colonequals\max\left\{2b_1b_2a_3-b_1b_2b_3, \; a_1b_2a_3+b_1a_2a_3-a_1a_2b_3,\right.\\
%    &\qquad  \qquad \qquad \qquad\quad\;\left.b_1a_2a_3+b_1b_2a_3-b_1a_2b_3, \; a_1b_2a_3+b_1b_2a_3-a_1b_2b_3\right\}  \\
%   & = a_1b_2a_3+b_1b_2a_3-a_1b_2b_3.
%   \end{align*}
%\end{lemma}

\begin{proof}[Proof of Lemma \ref{lem3}]
\begin{align*}
a_1b_2a_3+b_1b_2a_3-a_1b_2b_3 -(a_1b_2a_3+b_1a_2a_3-a_1a_2b_3)&\\
=(b_1a_3-a_1b_3)(b_2-a_2) &\geq0. \\
a_1b_2a_3+b_1b_2a_3-a_1b_2b_3 -(b_1a_2a_3+b_1b_2a_3-b_1a_2b_3)&\\
= (b_1a_2-a_1b_2)(b_3-a_3) &\geq0. \\
a_1b_2a_3+b_1b_2a_3-a_1b_2b_3 -(2b_1b_2a_3-b_1b_2b_3) = b_2(b_1-a_1)(b_3-a_3) &\geq0.
\end{align*}
The first two inequalities follow from Lemma \ref{lem10.1}.
\end{proof}

%\begin{lemma}
%\label{lem4}
%\begin{align*}
%z_4&\colonequals\max\left\{a_1b_2a_3+b_1a_2a_3-b_1b_2b_3, \;2a_1a_2a_3-a_1a_2b_3,\right.\\
%   &\qquad  \qquad \qquad \qquad\quad \left.a_1a_2a_3+b_1a_2a_3-b_1a_2b_3, \; a_1a_2a_3+a_1b_2a_3-a_1b_2b_3\right\}\\
%   & = 2a_1a_2a_3-a_1a_2b_3.
%   \end{align*}
%\end{lemma}

\begin{proof}[Proof of Lemma \ref{lem4}]
\begin{align*}
&2a_1a_2a_3-a_1a_2b_3 -(a_1a_2a_3+b_1a_2a_3-b_1a_2b_3)=a_2(b_1-a_1)(b_3-a_3) \geq0 \\
&2a_1a_2a_3-a_1a_2b_3 -(a_1a_2a_3+a_1b_2a_3-a_1b_2b_3)= a_1(b_2-a_2)(b_3-a_3) \geq0 \\
&2a_1a_2a_3-a_1a_2b_3 -(a_1b_2a_3+b_1a_2a_3-b_1b_2b_3) = \\ & \qquad \qquad \qquad\qquad\qquad\;\; (b_1-a_1)(b_2b_3-a_2a_3)+a_1(b_2-a_2)(b_3-a_3) \geq0.
\end{align*}
\end{proof}

And the proofs of the four lemmas used in calculating $V(Q,R,R)$:

%\begin{lemma}
%\label{lem5}
%\begin{align*}
% z_5&\colonequals\max\left\{b_1b_2a_3-b_1a_2b_3-b_1b_2b_3, \; a_1a_2a_3-a_1a_2b_3-b_1a_2b_3,\right.\\
%   &\qquad  \qquad \qquad \qquad\qquad\qquad\;\; \left.b_1a_2a_3-2b_1a_2b_3,\; a_1b_2a_3-a_1a_2b_3-b_1b_2b_3\right\}\\
%   &  = a_1a_2a_3-a_1a_2b_3-b_1a_2b_3.
%   \end{align*}
%\end{lemma}

\begin{proof}[Proof of Lemma \ref{lem5}]
\begin{align*}
a_1a_2a_3-a_1a_2b_3-b_1a_2b_3 -(b_1a_2a_3-2b_1a_2b_3)=a_2(b_1-a_1)(b_3-a_3) &\geq0 \\
a_1a_2a_3-a_1a_2b_3-b_1a_2b_3 -(a_1b_2a_3-a_1a_2b_3-b_1b_2b_3)=& \\
(b_1b_3-a_1a_3)(b_2-a_2) &\geq0 \\
a_1a_2a_3-a_1a_2b_3-b_1a_2b_3 -(b_1b_2a_3-b_1a_2b_3-b_1b_2b_3) = &\\
(b_1b_2-a_1a_2)(b_3-a_3) &\geq0.
\end{align*}
\end{proof}

%\begin{lemma}
%\label{lem6}
%\begin{align*}
% z_6&\colonequals\max\left\{b_1b_2a_3-a_1b_2b_3-b_1b_2b_3, \; a_1a_2a_3-a_1a_2b_3-a_1b_2b_3,\right.\\
%   &\qquad  \qquad \qquad \qquad\qquad\qquad\;\;\left.b_1a_2a_3-b_1b_2b_3-a_1a_2b_3, \;a_1b_2a_3-2a_1b_2b_3\right\} \\
%   & = a_1a_2a_3-a_1a_2b_3-a_1b_2b_3.
%   \end{align*}
%\end{lemma}

\begin{proof}[Proof of Lemma \ref{lem6}]
\begin{align*}
a_1a_2a_3-a_1a_2b_3-a_1b_2b_3 -(b_1a_2a_3-b_1b_2b_3-a_1a_2b_3)&\\
=(b_2b_3-a_2a_3)(b_1-a_1) &\geq0 \\
a_1a_2a_3-a_1a_2b_3-a_1b_2b_3 -(a_1b_2a_3-2a_1b_2b_3)= a_1(b_2-a_2)(b_3-a_3) &\geq0 \\
a_1a_2a_3-a_1a_2b_3-a_1b_2b_3 -(b_1b_2a_3-a_1b_2b_3-b_1b_2b_3)&\\
 = (b_1b_2-a_1a_2)(b_3-a_3) &\geq0.
\end{align*}
\end{proof}

%\begin{lemma}
%\label{lem7}
%\begin{align*}
% z_7&\colonequals\max\left\{2b_1b_2b_3-b_1b_2a_3,\; a_1b_2b_3+b_1a_2b_3-a_1a_2a_3,\right.\\
%   & \qquad  \qquad \qquad \qquad\quad\;\left.b_1a_2b_3+b_1b_2b_3-b_1a_2a_3, \;a_1b_2b_3+b_1b_2b_3-a_1b_2a_3\right\}\\
%   & = 2b_1b_2b_3-b_1b_2a_3.
%   \end{align*}
%\end{lemma}

\begin{proof}[Proof of Lemma \ref{lem7}]
\begin{align*}
&2b_1b_2b_3-b_1b_2a_3 -(a_1b_2b_3+b_1a_2b_3-a_1a_2a_3)\\
& \;\;\;\;\;\quad=(b_1-a_1)\left(b_2(b_3-a_3)+b_3(b_2-a_2)\right)+a_1(b_2-a_2)(b_3-a_3) \geq0\\
&2b_1b_2b_3-b_1b_2a_3 -(b_1a_2b_3+b_1b_2b_3-b_1a_2a_3)= b_1(b_2-a_2)(b_3-a_3) \geq0 \\
&2b_1b_2b_3-b_1b_2a_3 -(a_1b_2b_3+b_1b_2b_3-a_1b_2a_3) = b_2(b_1-a_1)(b_3-a_3) \geq0.
\end{align*}
\end{proof}

%\begin{lemma}
%\label{lem8}
%\begin{align*}
%z_8&\colonequals\max\left\{a_1b_2b_3+b_1a_2b_3-b_1b_2a_3, \; 2a_1a_2b_3-a_1a_2a_3,\right.\\
%   &\qquad \qquad \qquad \qquad\quad \left.a_1a_2b_3+b_1a_2b_3-b_1a_2a_3, \;a_1a_2b_3+a_1b_2b_3-a_1b_2a_3\right\}\\
%   &= a_1a_2b_3+b_1a_2b_3-b_1a_2a_3.
%   \end{align*}
%\end{lemma}

\begin{proof}[Proof of Lemma \ref{lem8}]
\begin{align*}
a_1a_2b_3+b_1a_2b_3-b_1a_2a_3 -(2a_1a_2b_3-a_1a_2a_3)=a_2(b_1-a_1)(b_3-a_3) &\geq0 \\
a_1a_2b_3+b_1a_2b_3-b_1a_2a_3 -(a_1a_2b_3+a_1b_2b_3-a_1b_2a_3)&\\
= (b_1a_2-a_1b_2)(b_3-a_3) &\geq0 \\
a_1a_2b_3+b_1a_2b_3-b_1a_2a_3 -(a_1b_2b_3+b_1a_2b_3-b_1b_2a_3)& \\
= (b_1a_3-a_1b_3)(b_2-a_2) &\geq0.
\end{align*}
The last two inequalities follow from Lemma \ref{lem10.1}.
\end{proof}



%\begin{thebibliography}{00}


%\bibitem{}

%\end{thebibliography}
\end{document}